\documentclass[11pt,reqno]{amsart}
\usepackage{amscd,amssymb,amsmath,amsthm}
\usepackage[arrow,matrix]{xy}
\usepackage{graphicx}
\usepackage{cite}
\newtheorem{thm}[subsection]{Theorem}
\newtheorem{lemma}[subsection]{Lemma}

\newtheorem{cor}[subsection]{Corollary}

\newtheorem{rk}[subsection]{Remark}

\numberwithin{equation}{section} 

\newcommand{\Z}{\mathbb{Z}}
\newcommand{\Q}{\mathbb{Q}}

\newcommand{\C}{\mathbb{C}}

\newcommand{\bc}{\begin{center}}
\newcommand{\ec}{\end{center}}
\newcommand{\be}{\begin{equation}}
\newcommand{\ee}{\end{equation}}
\newcommand{\bea}{\begin{eqnarray}}
\newcommand{\eea}{\end{eqnarray}}
\newcommand{\ba}{\begin{array}}
\newcommand{\ea}{\end{array}}

\newcommand{\dsf}{\displaystyle\frac}

\def\l{\lambda}
\def\g{\gamma}
\def\r{\rho}

\def\Q{\mathbb{Q}}

\def\Z{\mathbb{Z}}
\def\N{\mathbb{N}}
\def\C{\mathbb{C}}

\begin{document}

\title[$p$-adic $(2,1)$-rational dynamical systems]{$p$-adic $(2,1)$-rational dynamical systems}

\author{Albeverio S, Rozikov U.A., Sattarov I.A.}

 \address{S.\ Albeverio\\ Institut f\"ur Angewandte Mathematik and HCM, Universit\"at Bonn, Endenicher Allee 60, 53115 Bonn, Germany; CERFIM (Locarno).}
\email {albeverio@uni-bonn.de}

 \address{U.\ A.\ Rozikov and I.A. Sattarov\\ Institute of mathematics and information technologies,
29, Do'rmon Yo'li str., 100125, Tashkent, Uzbekistan.}
\email {rozikovu@yandex.ru\ \ iskandar1207@rambler.ru}

\begin{abstract}

We investigate the behavior of the trajectory of an arbitrary $(2,1)$-rational
$p$-adic dynamical system in a complex $p$-adic field $\C_p$. We study, in particular, Siegel disks and attractors for such dynamical systems. The set of fixed points
of the $(2,1)$-rational functions can be empty, or consist of one element or of two elements. We obtain the following results:

(i) In the case where there is no fixed point we show that the $p$-adic dynamical system
has a 2-periodic cycle $x_1, x_2$ which only can be either an attracting or an indifferent one.
If it is attracting
then it attracts each trajectory which starts from an element of a ball of radius $r=|x_1-x_2|_p$ with the center at $x_1$ or at $x_2$. If the 2-periodic cycle is an indifferent one, then in each step the above mentioned balls transfer to each other. All the other spheres with radius $>r$ and the center at
$x_1$ and $x_2$ are invariant independently of the attractiveness of the cycle.

(ii) In the case where the fixed point $x_0$ is unique we prove that if the point is attracting then there exists $\delta>0$, such that the basin of attraction for $x_0$ is the ball of radius $\delta$ and the center at $x_0$ and any sphere with radius $\geq \delta$ is invariant.
   If $x_0$ is an indifferent point then all spheres with the center at $x_0$ are invariant. If $x_0$ is a repelling point then
   there exits $\delta>0$, such that the trajectory which starts at an element of the ball of radius $\delta$ with the center in $x_0$ leaves this ball, whereas any sphere with radius $\geq \delta$ is invariant.

(iii) In case of existence of two fixed points, the $p$-adic dynamical system has a very rich behavior:
we show that Siegel disks may either coincide or be disjoint for
different fixed points of the dynamical system. Besides, we find
the basin of the attractor of the system. Varying the parameters it is proven that there exists an integer $k\geq 2$,
and spheres $S_{r_1}(x_i), \dots, S_{r_k}(x_i)$  such that the limiting trajectory will be periodically
traveling the spheres $S_{r_j}$. For some values of the parameters there are trajectories which go arbitrary far from the fixed points.

\end{abstract}

\maketitle

\section{Introduction}

In this paper we will state some results concerning discrete dynamical systems defined over the complex $p$-adic field $\C_p$. The interest in such systems and in the ways they
can be applied has been rapidly increasing during the last couple of decades (see, e.g., \cite{Sa} and the references therein).
The $p$-adic numbers were first introduced by the German
mathematician K.Hensel. For about a century after the discovery of
$p$-adic numbers, they were mainly considered objects of pure
mathematics. Beginning with 1980's various models described in the
language of $p$-adic analysis have been actively studied.  More
precisely, models over the field of $p$-adic numbers have been
considered which, based on the conception that $p$-adic numbers might
provide a more exact and more adequate description of micro-world
phenomena. Numerous applications of these numbers to theoretical
physics have been proposed in papers \cite{ADFV}, \cite{FW},
\cite{MP}, \cite{V1},\cite{V2} to quantum mechanics \cite{Kh1}, to
$p$-adic - valued physical observable \cite{Kh1} and many others
\cite{Kh2},\cite{VVZ}.

The study of $p$-adic dynamical systems arises in Diophantine
geometry in the constructions of canonical heights, used for
counting rational points on algebraic vertices over a number
field, as in \cite{CS}. In \cite{Kh4},\cite{TVW} $p$-adic field
have arisen in physics in the theory of superstrings, promoting
questions about their dynamics. Also some applications of $p$-adic
dynamical systems to some biological, physical systems has been
proposed in
\cite{ABKO}-\cite{AKTS},\cite{DGKS},\cite{Kh4},\cite{Kh5}. Other
studies of non-Archimedean dynamics in the neighborhood of a
periodic and of the counting of periodic points over global fields
using local fields appear in \cite{HY},\cite{L},\cite{P}. It is
known that the analytic functions play important role in complex
analysis. In the $p$-adic analysis the rational functions play a
similar role to the analytic functions in complex analysis
\cite{R}. Therefore, naturally one arises a question to study the
dynamics of these functions in the $p$-adic analysis. On the hand,
such $p$-adic dynamical systems appear while studying $p$-adic
Gibbs measures \cite{gmr}, \cite{GRR},\cite{MR4}, \cite{MR1}-\cite{MR3}. In \cite{B1},\cite{B2} dynamics on the
Fatou set of a rational function defined over some finite
extension of $\Q_p$ have been studied, besides, an analogue of
Sullivan's no wandering domains theorem for $p$-adic rational
functions which have no wild recurrent Julia critical points were
proved. In \cite{AKTS}  the behavior  of a $p$-adic dynamical
system $f(x)=x^n$ in the fields of $p$-adic numbers $\Q_p$ and
complex $p$-adic numbers $\C_p$ was investigated. Some ergodic
properties of that dynamical system have been considered in
\cite{GKL}.

In \cite{MR} the behavior of the trajectory of a rational
$p$-adic dynamical system in complex $p$-adic filed $\C_p$ is studied. It is
studied Siegel disks and attractors of such dynamical systems. It is
shown that Siegel disks may either coincide or be disjoint for
different fixed points. Besides, the basin of the attractor of the rational dynamical system is found. It is proved that such
kind of dynamical system is not ergodic on a unit sphere with
respect to the Haar measure.

The base of $p$-adic analysis, $p$-adic mathematical physics
are explained in \cite{G},\cite{Ko},\cite{VVZ}.

In this paper we investigate the behavior of trajectory of an arbitrary $(2,1)$-rational
$p$-adic dynamical system in complex $p$-adic field $\C_p$. The paper is organized as follows:
in Section 2 we give some preliminaries. Section 3 contains the definition of the concept of a $(2,1)$-rational function.
Section 4 is devoted to the $p$-adic dynamical system which has a unique fixed point $x_0$. We prove that if the point is attracting then there exits $\delta>0$, such that the basin of attraction for $x_0$ is the ball of radius $\delta$ and the center at $x_0$ and any sphere with radius $\geq \delta$ is invariant.
   If $x_0$ is an indifferent point then all spheres with the center at $x_0$ are invariant. If $x_0$ is a repelling point then
   there exits $\delta>0$, such that the trajectory which starts at an element of the ball of radius $\delta$ with center at $x_0$ leaves this ball, whereas any sphere with radius $\geq \delta$ is invariant. Section 5 contains results concerning
the $p$-adic dynamical systems which have no fixed point. We show that such $p$-adic dynamical systems have a 2-periodic cycle $x_1, x_2$ which only can be an attracting or an indifferent one.
If it is attracting then it attracts each trajectory which starts from an element of a ball of radius $r=|x_1-x_2|_p$ with the center at $x_1$ or at $x_2$. If the 2-periodic cycle is an indifferent one then in each step the above mentioned balls transfer to each other. All the other spheres with radius $>r$ and the center at
$x_1$ and $x_2$ are invariant independently on the attractiveness of the cycle.
The last section is devoted to the case of existence of two fixed points. We show, in particular, that Siegel disks may either coincide or be disjoint for different fixed points of the dynamical system. Besides, we find
the basin of the attractor of the system. Varying the parameters we prove that there exists $k\geq 2$,
and spheres $S_{r_1}(x_i), \dots, S_{r_k}(x_i)$  such that the limiting trajectory will be periodically
traveling the spheres $S_{r_j}(x_i)$. For some values of the parameters there are trajectories which go arbitrary far from the fixed points.

\section{Preliminaries}

\subsection{$p$-adic numbers}

Let $\Q$ be the field of rational numbers. The greatest common
divisor of the positive integers $n$ and $m$ is denotes by
$(n,m)$. Every rational number $x\neq 0$ can be represented in the
form $x=p^r\frac{n}{m}$, where $r,n\in\mathbb{Z}$, $m$ is a
positive integer, $(p,n)=1$, $(p,m)=1$ and $p$ is a fixed prime
number.

The $p$-adic norm of $x$ is given by
$$
|x|_p=\left\{
\begin{array}{ll}
p^{-r}, & \ \textrm{ for $x\neq 0$},\\[2mm]
0, &\ \textrm{ for $x=0$}.\\
\end{array}
\right.
$$
It has the following properties:

1) $|x|_p\geq 0$ and $|x|_p=0$ if and only if $x=0$,

2) $|xy|_p=|x|_p|y|_p$,

3) the strong triangle inequality holds
$$
|x+y|_p\leq\max\{|x|_p,|y|_p\},
$$

3.1) if $|x|_p\neq |y|_p$ then $|x+y|_p=\max\{|x|_p,|y|_p\}$,

3.2) if $|x|_p=|y|_p$ then $|x+y|_p\leq |x|_p$.

Thus $|x|_p$ is a non-Archimedean norm.

The completion of $\Q$ with  respect to the $p$-adic norm defines
the $p$-adic field which is denoted by $\Q_p$.

The well-known Ostrovsky's theorem asserts that the norms $|x|=|x|_{\infty}$ and
$|x|_p$, $p\in {\mathrm P}$,(where ${\mathrm P} =\{2,3,5...\}$ denotes the set of prime numbers) exhaust all nonequivalent norms on $\Q$ (see \cite{Ko}).
Any $p$-adic number $x\neq 0$ can be uniquely represented by the
canonical series, convergence in the $|x|_p$-norm:
$$
x=p^{\gamma(x)}(x_0+x_1p+x_2p^2+...) ,
\eqno(2.1)
$$
where $\gamma=\gamma(x)\in\mathbb Z$ and $x_j$ are integers, $0\leq x_j\leq p-1$, $x_0>0$,
$j\in \N_0=\{0,1,2,...\}$ (see for more details \cite{G},\cite{Ko}).
Observe that in this case $|x|_p=p^{-\g(x)}$.

The algebraic completion of $\Q_p$ which is complete with respect to $|\cdot|_p$ is denoted by $\C_p$ and it is called
{\it the field of complex $p$-adic numbers}.  For any $a\in\C_p$ and $r>0$ denote
$$
U_r(a)=\{x\in\C_p : |x-a|_p\leq r\},\ \ V_r(a)=\{x\in\C_p : |x-a|_p< r\},
$$
$$
S_r(a)=\{x\in\C_p : |x-a|_p= r\}.
$$

A function $f:U_r(a)\to\C_p$ is said to be {\it analytic} if it can be represented by
$$
f(x)=\sum_{n=0}^{\infty}f_n(x-a)^n, \ \ \ f_n\in \C_p,
$$ which converges uniformly on the ball $U_r(a)$.

\begin{thm}\label{tt}\cite{VVZ} Let $f(x)$ be an analytic function on a ball $U_r(a)$ and $f'(a)\ne 0$, $|f'(a)|_p=p^n$. Then there exists a ball $U_\rho(a)$, with $\rho<r$, such that $f$ is one-to-one map from $U_\rho(a)$ into $U_{\rho+n}(b)$ with $b=f(a)$, and the inverse function $g(y)$ is an analytic function on $U_{\rho+n}(b)$ and the following equality holds
$$g'(b)={1\over f'(a)}.$$
\end{thm}

\subsection{Dynamical systems in $\C_p$}

In this section we recall some known facts concerning dynamical systems $(f,U)$ in $\C_p$,
where $f: x\in U\to f(x)\in U$ is an analytic function and $U=U_r(a)$ or $\C_p$.

Now let $f:U\to U$ be an analytic function. Denote $x_n=f^n(x_0)$, where $x_0\in U$ and
$f^n(x)=\underbrace{f\circ\dots\circ f}_n(x)$.

Let us first recall some  the standard terminology of the theory of dynamical
systems (see for example \cite{PJS}). If $f(x_0)=x_0$ then $x_0$
is called a {\it fixed point}. The set of all fixed points of $f$ is denoted by Fix$(f)$. A fixed point $x_0$ is called an
{\it attractor} if there exists a neighborhood $V(x_0)$ of $x_0$
such that for all points $y\in V(x_0)$ it holds that
$\lim\limits_{n\to\infty}y_n=x_0$. If $x_0$ is an attractor then
its {\it basin of attraction} is
$$
A(x_0)=\{y\in \C_p :\ y_n\to x_0, \ n\to\infty\}.
$$
A fixed point $x_0$ is called {\it repeller} if there  exists a
neighborhood $V(x_0)$ of $x_0$ such that $|f(x)-x_0|_p>|x-x_0|_p$
for $x\in V(x_0)$, $x\neq x_0$. Let $x_0$ be a fixed point of a
function $f(x)$. The ball $V_r(x_0)$ (contained in $U$) is said to
be a {\it Siegel disk} if each sphere $S_{\r}(x_0)$, $\r<r$ is an
invariant sphere for $f(x)$, i.e. if $x\in S_{\r}(x_0)$ then for all
iterated points we have $x_n\in S_{\r}(x_0)$, for all $n\in \N$.  The
union of all Siegel disks with the center at $x_0$ is said to {\it
a maximum Siegel disk} and is denoted by $SI(x_0)$.

 In complex geometry, the center of a disk
is uniquely determined by the disk, and different fixed points
cannot have the same Siegel disks. In non-Archimedean geometry, a
center of a disk is a point which belongs to the disk.
Therefore, different fixed points may have the same
Siegel disk \cite{AKTS}.

Let $x_0$ be a fixed point of an analytic function  $f(x)$. Put
$$
\l=\frac{d}{dx}f(x_0).
$$

The point $x_0$ is {\it attractive} if $0\leq |\l|_p<1$, {\it indifferent} if
$|\l|_p=1$, and {\it repelling} if $|\l|_p>1$.

\begin{thm}\label{t2.1}\cite{AKTS}  Let $x_0$ be a fixed point of an analytic function $f:U\to U$.
The following assertions hold
\begin{itemize}
\item[1.] if $x_0$ is an attractive point of $f$, then it is an attractor
of the dynamical system $(f,U)$. If $r>0$ satisfies the inequality
$$
q=\max_{1\leq n<\infty}\bigg|\frac{1}{n!}\frac{d^nf}{dx^n}(x_0)\bigg|_pr^{n-1}<1
\eqno(2.2)
$$
and $U_r(x_0)\subset U$ then $U_r(x_0)\subset A(x_0)$;\\

\item[2.] if $x_0$ is an indifferent point of $f$ then it is the center of a Siegel disk. If $r$
satisfies the inequality
$$
s=\max_{2\leq n<\infty}\bigg|\frac{1}{n!}\frac{d^nf}{dx^n}(x_0)\bigg|_pr^{n-1}<|f'(x_0)|_p
\eqno(2.3)
$$
and $U_r(x_0)\subset U$ then $U_r(x_0)\subset SI(x_0)$;\\

\item[3.] if $x_0$ is a repelling point of $f$ then $x_0$ is a repeller of the dynamical
system $(f,U)$.
\end{itemize}
\end{thm}

\section{$(2,1)$-Rational $p$-adic dynamical systems}

A function is called an $(n,m)$-rational function if and only if
it can be written in the form $f(x)={P_n(x)\over Q_m(x)}$,
where $P_n(x)$ and $Q_m(x)$ are polynomial functions with degree $n$ and
$m$ respectively ($Q_m(x)$ is not the zero polynomial).

In this paper we consider the dynamical system associated with the $(2,1)$-rational function $f:\C_p\to\C_p$ defined by
\begin{equation}\label{3.1}
f(x)=\frac{x^2+ax+b}{cx+d}, \ \ a,b,c,d\in \C_p,\ \  c\neq 0,\  \ d^2-acd+bc^2\ne 0,
\end{equation}

where  $x\neq \hat x=-\dsf{d}{c}$.

Note that any $(2,1)$-rational function can be written in the form (\ref{3.1}). If we take $a=b=0$ then the function coincides with the function considered in \cite{KM}.
But the authors of \cite{KM} did not consider the case $c=1, d\ne 0$ in which case the function has no fixed point. In this paper we consider the following cases:

1) $c=1$, $a\ne d$. In this case the function $f$ has a unique fixed point $x_0={b\over d-a}$.

2) $c=1$, $a=d$, $b\ne 0$. In this case there is no fixed point.

3) $c=1$, $a=d$, $b=0$. In this case $f(x)=x$, i.e. the function $f$ is the id-function.

4) $c\ne 1$. There are two fixed points:
\begin{equation}\label{3.2}
x_{1,2}=\frac{a-d\pm\sqrt{(a-d)^2+4(c-1)b}}{2(c-1)}.
\end{equation}

The function $f$ can be written in the following form
$$ f(x)={1\over c^3}\left(c^2x+ac^2-dc+{d^2-acd+bc^2\over x+{d\over c}}\right).$$
Using this representation we get the following formulas:
\begin{equation}\label{f'}
f'(x)={1\over c^3}\left(c^2-{d^2-acd+bc^2\over \left(x+{d\over c}\right)^2}\right),
\end{equation}

\begin{equation}\label{f''}
{d^n\over d^nx}f(x)=(-1)^n n!{d^2-acd+bc^2\over c^3\left(x+{d\over c}\right)^{n+1}},  \ \ n\geq 2.
\end{equation}

\section{The case of a unique fixed point}
In this section we consider the case $c=1$ and $a\ne d$. As mentioned above in this case there is a unique fixed point: $x_0={b\over d-a}$.

Denote
$$\mathcal P_1=\{x\in \C_p: \exists n\in \N\cup\{0\}, f^n(x)=\hat x\}, \ \ \mbox{for} \ \ c=1.$$

For $c=1$ we have
\begin{equation}\label{1}
    f'(x_0)={ad+b-a^2\over d^2-ad+b}.
\end{equation}

\subsection{The case $|f'(x_0)|_p=1$.}

If $|f'(x_0)|_p=1$ then by Theorem \ref{t2.1} the point $x_0$ is an indifferent point and it is the center of a Siegel disk.

\begin{thm}\label{t1.1}
If $c=1$, $|f'(x_0)|_p=1$ then $SI(x_0)=\C_p\setminus {\mathcal P}_1$.
\end{thm}
\begin{proof}
Denote $\delta=|x_0+d|_p$.
We shall prove that $f(S_r(x_0)\setminus {\mathcal P}_1)\subset S_r(x_0)$ for any $r>0$.
Take an arbitrary $y\in S_r(x_0)\setminus {\mathcal P}_1$, i.e., $y=x_0+\gamma$, $|\gamma|_p=r$.
We have
\begin{equation}\label{e1}
|f(y)-x_0|_p=r\cdot\left|{x^2_0+2dx_0+ad-b+\gamma(x_0+d)\over (x_0+d)^2+\gamma(x_0+d)}\right|_p=r\cdot{\left|f'(x_0)+{\gamma\over x_0+d}\right|_p\over\left|1+{\gamma\over x_0+d}\right|_p}.
\end{equation}

For $r\ne \delta$ using property 3.1) of the $p$-adic norm we get $|f(y)-x_0|_p=r$, i.e. $f(y)\in S_r(x_0)$. In case $r=\delta$ one can use the $p$-adic version of the inverse function given by Theorem \ref{tt} and show that $f(S_\delta(x_0)\setminus {\mathcal P}_1)\subset S_\delta(x_0)$. Indeed, if we assume that there exists $x\in S_\delta(x_0)$ such that $y=f(x)\notin S_\delta(x_0)$ then this by the inverse function theorem gives the following contradiction:
$$\delta=|x-x_0|_p=|f^{-1}(y)-x_0|_p=\left|{1\over f'(x_0)}\right|_p|y-x_0|_p=|y-x_0|_p\ne \delta.$$
\end{proof}
\subsection{The case $|f'(x_0)|_p<1$.}
For $|f'(x_0)|_p<1$ the point $x_0$ is an attractive point of $f$.
\begin{thm}\label{t1.2}
If $c=1$, $|f'(x_0)|_p<1$ then
\begin{itemize}
\item[(i)] for any $\mu>\delta$ the sphere $S_\mu(x_0)$ is invariant with respect to $f$, i.e. $f(S_\mu(x_0))\subset S_\mu(x_0)$.

\item[(ii)] $V_\delta(x_0)\subset A(x_0)$.
\end{itemize}
\end{thm}
\begin{proof} (i)
Take an arbitrary $y\in S_\mu(x_0)$, i.e., $y=x_0+\gamma$, $|\gamma|_p=\mu$.
We have
\begin{equation}\label{e1}
|f(y)-x_0|_p=\mu\cdot{\left|f'(x_0)+{\gamma\over x_0+d}\right|_p\over\left|1+{\gamma\over x_0+d}\right|_p}.
\end{equation}

Using $\mu>\delta$ and property 3.1) of the $p$-adic norm we get $|f(y)-x_0|_p=\mu$, i.e. $f(y)\in S_\mu(x_0)$.

(ii)  We shall use Theorem \ref{t2.1}. By formula (\ref{f''}), for $c=1$, we have
$$q=\max_{1\leq n<\infty}\left|{1\over n!}\cdot\frac{d^nf}{dx^n}(x_0)\right|_pr^{n-1}=
\max_{1\leq n<\infty}\left|\frac{d^2-ad+b}{(x_0+d)^{n+1}}\right|_pr^{n-1}=$$
$$\max_{1\leq n<\infty}\left|\frac{d^2-ad+b}{(x_0+d)^{2}}\right|_p\left({r\over |x_0+d|_p}\right)^{n-1}=\max_{1\leq n<\infty}\left|1-f'(x_0)\right|_p\left({r\over |x_0+d|_p}\right)^{n-1}= $$ $$ \max_{1\leq n<\infty}\left({r\over |x_0+d|_p}\right)^{n-1}<1, \ \ \mbox{if} \ \ r<\delta.$$
Thus by Theorem \ref{t2.1} we get $V_\delta(x_0)\subset A(x_0)$.
\end{proof}
Hence we proved that all elements of $S_\mu(x_0)$ for $\mu<\delta$ are points of the basin of attraction $A(x_0)$, and all spheres $S_\mu(x_0)$ for $\mu>\delta$ are invariant. Moreover, since on $S_\mu(x_0)$ for $\mu>\delta$ there is no a fixed point of $f$ the trajectory $x_n=f^{n}(x)$ does not converge for any $x\in S_\mu(x_0)$, for $\mu>\delta$.
Now it remains to study the dynamical system for $x\in S_\delta(x_0)$.

\begin{lemma} $\mathcal P_1\subset S_\delta(x_0)$.
\end{lemma}
\begin{proof} First we note that $\hat x\in S_\delta(x_0)$. Indeed, for $c=1$ we have
$$|\hat x-x_0|_p=|-d-x_0|_p=|x_0+d|_p=\delta.$$ Now take an arbitrary $x\in \mathcal P_1$, $x\ne\hat x$ then $x$ is not in $S_\mu(x_0)$, $\forall \mu>\delta$ since the spheres are invariant with respect to $f$, so there is no $n$ with $f^n(x)=\hat x$. Moreover, $x\notin V_\delta(x_0)$, since $V_\delta(x_0)$ is subset of the attractor. Thus the only possibility is that $x\in S_\delta(x_0)$.
\end{proof}

\begin{thm}\label{t1.3} If $c=1$, $|f'(x_0)|_p<1$ and $x\in S_\delta(x_0)\setminus \mathcal P_1$ then there exist the following two possibilities:
\begin{itemize}
\item[1)] There exists $k\in\N$ and $\mu_k>\delta$ such that
$f^m(x)\in S_{\mu_k}(x_0)$, for any $m\geq k$.

\item[2)] The trajectory $\{f^k(x), k\geq 1\}$ is a subset of $S_\delta(x_0)$.
\end{itemize}
\end{thm}
\begin{proof} Take $x\in S_\delta(x_0)\setminus \mathcal P_1$ then $x=x_0+\gamma$ with $|\gamma|_p=\delta$. We have
\begin{equation}\label{e2}
|f(y)-x_0|_p=\delta\cdot{\left|f'(x_0)+{\gamma\over x_0+d}\right|_p\over\left|1+{\gamma\over x_0+d}\right|_p}={\delta\over\left|1+{\gamma\over x_0+d}\right|_p}\geq\delta.
\end{equation}
If $|f(x)-x_0|_p>\delta$ then there is $\mu_1>\delta$ such that $f(x)\in S_{\mu_1}(x_0)$. So in this case $k=1$. If $|f(x)-x_0|_p=\delta$ then we consider the following equality
\begin{equation}\label{d}
|f^2(x)-x_0|_p={|f(x)-x_0|_p\left|f'(x_0)+{f(x)-x_0\over x_0+d}\right|_p\over\left|1+{f(x)-x_0\over x_0+d}\right|_p}={\delta^2\over |f(x)+d|_p}.
\end{equation}
Since $f(x)\in S_\delta(x_0)$ we have $f(x)=x_0+\gamma_1$, with $|\gamma_1|_p=\delta$.
From (\ref{d}) we get
$$|f^2(x)-x_0|_p={\delta^2\over |\gamma_1+(x_0+d)|_p}\geq \delta.$$
Now, if $|f^2(x)-x_0|_p>\delta$ then there is $\mu_2>\delta$ such that $f^2(x)\in S_{\mu_2}(x_0)$. So in this case $k=2$. If $|f^2(x)-x_0|_p=\delta$ then we can continue the argument and get the following equality
$$|f^k(x)-x_0|_p={\delta^2\over |f^{k-1}(x)+d|_p}.$$
Hence in each step we may have two possibilities:  $|f^k(x)-x_0|_p=\delta$ or $|f^k(x)-x_0|_p>\delta$.
In case $|f^k(x)-x_0|_p>\delta$ there exists $\mu_k$ such that $f^k(x)\in S_{\mu_k}(x_0)$, and since  $S_{\mu_k}(x_0)$ is an invariant with respect to $f$ we get  $f^m(x)\in S_{\mu_k}(x_0)$ for any $m\geq k$. If $|f^k(x)-x_0|_p=\delta$ for any $k\in \N$ then $\{f^k(x), k\geq 1\}\subset S_\delta(x_0)$.
\end{proof}

From Theorems \ref{t1.1}-\ref{t1.3} we get immediately the following

\begin{cor}  $A(x_0)=V_\delta(x_0)$.
\end{cor}
\subsection{The case  $|f'(x_0)|_p>1$.}
If $|f'(x_0)|_p>1$ then the point $x_0$ is a repeller of the dynamical
system.

\begin{thm}\label{t1.4} If $c=1$, $|f'(x_0)|_p=q>1$ and $x\in S_\mu(x_0)\setminus \mathcal P_2$ then the following properties hold
\begin{itemize}
\item[(a)] If $\mu<\delta$ then
\begin{itemize}
\item[(a.1)] If $\mu q^m\ne \delta$ for all $m\in \N$ then there exists $k\in \N$ such that $\mu q^k>\delta$ and
$f^n(x)\in S_{\mu q^n}(x_0),\ \ n=1,\dots,k; \ \ f^{k+1}(x)\in S_{\delta q}(x_0); \ \ f^{k+2}(x)\in S_{\nu}(x_0), \ \ \mbox{where} \ \ \nu\leq \delta q,$
i.e. the trajectory after $k$ steps leaves $U_\delta(x_0)$.

\item[(a.2)] If $\mu q^m= \delta$ for some $m\in \N$ then
$f^n(x)\in S_{{\delta\over q^{m-n}}}(x_0),\ \ n=1,\dots,m; \ \ f^{m+1}(x)\in S_{\nu}(x_0), \ \ \mbox{where} \ \ \nu\geq \delta q.$
\end{itemize}

\item[(b)] If $\mu>\delta$ then
\begin{itemize}
\item[(b.1)] If $\mu>\delta q$ then $f(x)\in S_\mu(x_0)$;

\item[(b.2)] If $\mu<\delta q$ then $f(x)\in S_{\delta q}(x_0)$;

\item[(b.3)] If $\mu=\delta q$ then $f(x)\in S_{\nu}(x_0)$ where $\nu\leq \delta q$.
\end{itemize}
\end{itemize}
\end{thm}
\begin{proof} (a.1) For $x\in S_\mu(x_0)$ we have
\begin{equation}\label{e6}
|f(x)-x_0|_p=\mu\cdot{\left|f'(x_0)+{x-x_0\over x_0+d}\right|_p\over\left|1+{x-x_0\over x_0+d}\right|_p}=\mu q.
\end{equation}
If $\mu q>\delta$ then $k=1$. If $\mu q<\delta$ then we take $\mu_1=\mu q$ and get
 \begin{equation*}
|f^2(x)-x_0|_p=\mu_1\cdot{\left|f'(x_0)+{f(x)-x_0\over x_0+d}\right|_p\over\left|1+{f(x)-x_0\over x_0+d}\right|_p}= \mu_1 q=\mu q^2.
\end{equation*}
If $\mu q^2>\delta$ then $k=2$. Otherwise we take $\mu_2=\mu q^2$ and repeat the argument. Since $q>1$, iterating the argument we find a finite $k\in \N$ such that $\mu q^{k-1}<\delta$, $\mu q^k>\delta$ and
$$|f^k(x)-x_0|_p=\mu q^k.$$
Since $1<{\mu q^k\over \delta}<q$, using property 3.1) we obtain
\begin{equation*}
|f^{k+1}(x)-x_0|_p=\mu q^k\cdot{\left|f'(x_0)+{f^k(x)-x_0\over x_0+d}\right|_p\over\left|1+{f^k(x)-x_0\over x_0+d}\right|_p}= \mu q^k\cdot {q\over {\mu q^k\over \delta}}=\delta q.
\end{equation*}
For $k+2$ we have
\begin{equation*}
|f^{k+2}(x)-x_0|_p=\delta q\cdot{\left|f'(x_0)+{f^{k+1}(x)-x_0\over x_0+d}\right|_p\over\left|1+{f^{k+1}(x)-x_0\over x_0+d}\right|_p}\leq \delta q\cdot {q\over  q}=\delta q.
\end{equation*} Thus there is $\nu\leq \delta q$ such that $f^{k+2}(x)\in S_\nu(x_0)$.

(a.2) This proof up to the step $m$ similar to the proof of (a.1). Since $|f^m(x)-x_0|_p=\delta$, for
 $m+1$ we have
$$|f^{m+1}(x)-x_0|_p=\delta\cdot{\left|f'(x_0)+{f^{m}(x)-x_0\over x_0+d}\right|_p\over\left|1+{f^{m}(x)-x_0\over x_0+d}\right|_p}\geq \delta q.$$

(b.1) For $x\in S_\mu(x_0)$ using conditions of the assertion (b.1) we get
\begin{equation*}
|f(x)-x_0|_p=\mu\cdot{\left|f'(x_0)+{x-x_0\over x_0+d}\right|_p\over\left|1+{x-x_0\over x_0+d}\right|_p}=\mu\cdot {\mu/q\over \mu/q}=\mu.
\end{equation*}
The proofs of (b.2) and (b.3) are similar to the above proofs.
\end{proof}
\begin{rk} Theorem \ref{t1.4} gives the following character of the dynamical system when $x_0$ is a repeller point: the trajectory of a point from the inner of $U_{\delta q}(x_0)$ goes forward to the sphere $S_{\delta q}$ as soon as the trajectory reaches the sphere, in the next step, it may go back to the inner of $U_{\delta q}(x_0)$ or stay in $S_{\delta q}$ for some time and then go back to the inner of $U_{\delta q}(x_0)$. As soon as the trajectory goes outside of $U_{\delta q}(x_0)$ it will stay (all the rest time) in the sphere (outside of $U_{\delta q}(x_0)$) where first it came.
\end{rk}
\section{The case where there is no fixed point}

In this section we consider the case $c=1$, $a=d$, $b\ne 0$. In this case the function $f(x)={x^2+ax+b\over x+a}$ has no fixed point.
In such a case it will be interesting to find periodic points of $f$. Let us consider 2-periodic points, i.e. consider the equation
\begin{equation}\label{g}
g(x)\equiv f(f(x))=x+{b\over x+a}+{b(x+a)\over (x+a)^2+b}=x.
\end{equation}

This equation is equivalent to $(x+a)^2=-{b\over 2}$, which  has two solutions $t_{1,2}=-a\pm\sqrt{-b/2}$ in $\C_p$.
It is a surprise that $g'(t_1)=g'(t_2)=9$, i.e. the value does not depend on parameters $a$ and $b$. Thus we have
\begin{equation}
|g'(t_1)|_p=|g'(t_2)|_p=\left\{\begin{array}{ll}
1, \ \ \mbox{if} \ \ p\ne 3,\\[2mm]
1/9, \ \ \mbox{if} \ \ p=3.\\
\end{array}
\right.
\end{equation}

Note that the function $g$ (see (\ref{g})) is defined in $\C_p\setminus\{\hat x=-a, {\hat{\hat x}}_{\pm}\}$, where ${\hat{\hat x}}_{\pm}=-a\pm \sqrt{-b}$ are the solutions to the equation $f(x)=\hat x$.

Denote
$$\mathcal P_2=\{x\in \C_p: \exists n\in \N, \
\ {\rm such \, that} \  \ f^n(x)\in\{\hat x, \hat{\hat x}_-, \hat{\hat x}_+\}\},$$
$$h=|t_1+a|_p=|t_2+a|_p.$$

The following equalities are obvious:

 $$|t_1-t_2|_p=h, \ \ p\ne 2; \ \ |\hat x-t_1|_p=|\hat x-t_2|_p=h. $$

\subsection{Case $p\ne 3$.}
In this case each fixed point $t_1, t_2$ of $g$ is an indifferent point
 and is the center of a Siegel disk.

\begin{thm}\label{tg1} If $p\ne 3$ then $f(S_r(t_1)\setminus \mathcal P_2)\subseteq S_r(t_2)$,   $f(S_r(t_2)\setminus \mathcal P_2)\subseteq S_r(t_1)$, for any $r>0$.
\end{thm}
\begin{proof} We shall use the following equalities:
$$f(t_1)=t_2, \ \ f(t_2)=t_1; \ \ f'(t_1)=f'(t_2)=3.$$
Let $x\in S_r(t_1)\setminus\mathcal P_2$, i.e., $x=t_1+\gamma$ with $|\gamma|_p=r$. We have
\begin{equation}\label{ee}
|f(x)-t_2|_p=|f(x)-f(t_1)|_p=
r\cdot\left|{3+{\gamma\over t_1+a}\over 1+{\gamma\over t_1+a}}\right|_p.
\end{equation}
If $r\ne h$ then using the property 3.1) we get from (\ref{ee}) that $f(x)\in S_r(t_2)$.
In case $r=h$ we use the $p$-adic version of the inverse function given by Theorem \ref{tt}: assume that there exists $x\in S_h(t_1)$ such that $y=f(x)\notin S_h(t_2)$, then this by the inverse function theorem gives the following contradiction:
$$h=|x-t_1|_p=|f^{-1}(y)-f^{-1}(t_2)|_p=\left|{1\over f'(t_2)}\right|_p|y-t_2|_p=|y-t_2|_p\ne h.$$
\end{proof}

\subsection{Case $p=3$.}
In this case each fixed point $t_1, t_2$ of $g$ is an attractive point of $g$.

\begin{thm}\label{tg2} If $p=3$ then
\begin{itemize}
\item[(a)] If $r<h$ then for any $x\in S_r(t_1)\setminus \mathcal P_2$,
$\lim_{n\to \infty}f^{2n}(x)=t_1$ and $\lim_{n\to \infty}f^{2n+1}(x)=t_2$; for any $x\in S_r(t_2)\setminus \mathcal P_2$,
$\lim_{n\to \infty}f^{2n}(x)=t_2$ and $\lim_{n\to \infty}f^{2n+1}(x)=t_1$.\\

\item[(b)] If $r>h$ then $f(S_r(t_1)\setminus \mathcal P_2)\subseteq S_r(t_2)$,   $f(S_r(t_2)\setminus \mathcal P_2)\subseteq S_r(t_1)$.\\

\item[(c)] If $r=h$ then for any $x\in S_h(t_1)\setminus\mathcal P_2$ there exists $\nu=\nu(x)\geq h$ such that
$f^{2n}(x)\in S_\nu(t_1)$ and  $f^{2n+1}(x)\in S_\nu(t_2)$; for any $y\in S_h(t_2)\setminus P_2$ there exists $\mu=\mu(y)\geq h$ such that
$f^{2n}(y)\in S_\mu(t_2)$ and  $f^{2n+1}(y)\in S_\mu(t_1)$.\\
\end{itemize}
\end{thm}

\begin{proof} Let $x\in S_r(t_1)\setminus\mathcal P_2$, i.e., $x=t_1+\gamma$ with $|\gamma|_3=r$. We have
\begin{equation}\label{ee1}
|f(x)-t_2|_3=|f(x)-f(t_1)|_3=
r\cdot\left|{3+{\gamma\over t_1+a}\over 1+{\gamma\over t_1+a}}\right|_3=\varphi(r)=\left\{\begin{array}{lllll}
r, \ \ \ \ \ \mbox{if} \ \ r>h,\\[2mm]
\geq r, \ \ \mbox{if} \ \ r=h,\\[2mm]
{r^2\over h}, \ \ \ \ \mbox{if} \ \ {h\over 3}<r<h,\\[2mm]
\leq {r\over 3}, \ \ \mbox{if} \ \ r={h\over 3},\\[2mm]
{r\over 3}, \ \ \ \ \ \mbox{if} \ \ r<{h\over 3}.
\end{array}\right.
\end{equation}
For $f^2(x)$ we have
$$
|f^2(x)-t_1|_3=|f^2(x)-f^2(t_1)|_3=
|f(x)-t_2|_3\cdot\left|{3+{f(x)-t_2\over t_2+a}\over 1+{f(x)-t_2\over t_2+a}}\right|_3=$$ $$ \varphi(\varphi(x))=\left\{\begin{array}{lllll}
\varphi(r), \ \ \ \ \ \mbox{if} \ \ \varphi(r)>h,\\[2mm]
\geq \varphi(r), \ \ \mbox{if} \ \ \varphi(r)=h,\\[2mm]
{(\varphi(r))^2\over h}, \ \  \  \mbox{if} \ \ {h\over 3}<\varphi(r)<h,\\[2mm]
\leq {\varphi(r)\over 3}, \ \ \mbox{if} \ \ \varphi(r)={h\over 3},\\[2mm]
{\varphi(r)\over 3}, \ \ \ \ \ \mbox{if} \ \ \varphi(r)<{h\over 3}.
\end{array}\right.
$$
Iterating this argument we obtain the following  formulas for $x\in S_r(t_1)\setminus \mathcal P_2$:
\begin{equation}\label{ef}
    |f^{2n}(x)-t_1|_3=\varphi^{2n}(r), \ \ |f^{2n+1}(x)-t_2|_3=\varphi^{2n+1}(r).
\end{equation}
Thus the dynamics of the radius $r$ of the spheres is given by the function $\varphi:[0,+\infty)\to [0,+\infty)$, which is defined in formula (\ref{ee1}). The following properties of $\varphi$ are obvious:
\begin{itemize}
  \item[$1^o$.] The set of fixed points of $\varphi(x)$ is Fix$(\varphi)=\{0\}\cup(h,+\infty)\cup \{h: \, {\rm if}\, \varphi(h)=h\}$;

   \item[$2^o$.] If $\varphi^n(h)=h$, $\forall n=1,\dots,m$, for some $m\in\N$ and $\varphi^{m+1}(h)=h_*>h$ then $\varphi^k(h)=h_*$ for all $k\geq m+1$;

   \item[$3^o$.] The fixed point $r=0$ is attractive with basin of attraction $[0,h)$, independently on the value $\varphi({h\over 3})\leq {h\over 9}$.
   \end{itemize}
Now using (\ref{ef}) it is easy to see that the assertion (a) follows from property $3^o$; the assertion (b) follows from property $1^o$ and (c) follows from $2^o$.
      \end{proof}

\section{ The case with two fixed points}

In this section we consider the case $c\ne 1$, $c\ne 0$, $d^2-acd+bc^2\ne 0$ then the function $f$ has two fixed points $x_1$ and $x_2$ (see (\ref{3.2})). We denote

$$\mathcal P_3=\{x\in \C_p: \exists n\in \N\cup\{0\}, f^n(x)=\hat x\}, \ \ \mbox{for} \ \ c\ne 1.$$

For any $x\in \C_p$, $x\ne \hat x$, by simple calculations we get
\begin{equation}\label{ff}
    |f(x)-x_i|_p={|x-x_i|_p\over |c|_p}\cdot{|\alpha(x_i)+x-x_i|_p\over |\beta(x_i)+x-x_i|_p}, \ \ i=1,2,
\end{equation}
where
$$\alpha(x)={cx^2+2dx+ad-bc\over cx+d}, \ \ \beta(x)={cx+d\over c}.$$
Denote
$$\alpha_i=|\alpha(x_i)|_p, \ \ \beta_i=|\beta(x_i)|_p, \ \ i=1,2.$$
Consider the following functions:

For $0\leq \alpha<\beta$ define the function $\varphi_{\alpha,\beta}: [0,+\infty)\to [0,+\infty)$ by
$$\varphi_{\alpha,\beta}(r)={1\over |c|_p}\left\{\begin{array}{lllll}
{\alpha\over\beta}r, \ \ {\rm if} \ \ r<\alpha,\\[2mm]
\alpha^*, \ \ {\rm if} \ \ r=\alpha,\\[2mm]
{r^2\over\beta}, \ \ {\rm if} \ \ \alpha<r<\beta,\\[2mm]
\beta^*, \ \ {\rm if} \ \ r=\beta,\\[2mm]
r, \ \ \ \ {\rm if} \ \ r>\beta,
\end{array}
\right.
$$
where $\alpha^*$ and $\beta^*$ are some given numbers with $\alpha^*\leq{\alpha^2\over\beta}$, $\beta^*\geq\beta$.

For $0\leq \beta<\alpha$ define the function $\phi_{\alpha,\beta}: [0,+\infty)\to [0,+\infty)$ by
$$\phi_{\alpha,\beta}(r)={1\over |c|_p}\left\{\begin{array}{lllll}
{\alpha\over\beta}r, \ \ {\rm if} \ \ r<\beta,\\[2mm]
\beta', \ \ {\rm if} \ \ r=\beta,\\[2mm]
\alpha, \ \ {\rm if} \ \ \beta<r<\alpha,\\[2mm]
\alpha', \ \ {\rm if} \ \ r=\alpha,\\[2mm]
r, \ \ \ \ {\rm if} \ \ r>\alpha,
\end{array}
\right.
$$
where $\alpha'$ and $\beta'$ some positive numbers with $\alpha'\leq \alpha$, $\beta'\geq\alpha$.

For $\alpha\geq 0$ we define the function $\psi_{\alpha}: [0,+\infty)\to [0,+\infty)$ by
$$\psi_{\alpha}(r)={1\over |c|_p}\left\{\begin{array}{ll}
r, \ \ {\rm if} \ \ r\ne\alpha,\\[2mm]
\hat\alpha, \ \ {\rm if} \ \ r=\alpha,\\[2mm]
\end{array}
\right.
$$
where $\hat\alpha$ is a given number.

Using the formula (\ref{ff}) we easily get the following:

\begin{lemma}\label{lf} If $x\in S_r(x_i)$, then the following formula holds
$$|f^n(x)-x_i|_p=\left\{\begin{array}{lll}
\varphi_{\alpha_i,\beta_i}^n(r), \ \ \mbox{if} \ \ \alpha_i<\beta_i,\\[2mm]
\phi_{\alpha_i,\beta_i}^n(r), \ \ \mbox{if} \ \ \alpha_i>\beta_i,\\[2mm]
\psi_{\alpha_i}^n(r), \ \ \ \ \mbox{if} \ \ \alpha_i=\beta_i.
\end{array}\right.  \ \ n\geq 1,\ \ i=1,2.$$
\end{lemma}
Thus the $p$-adic dynamical system $f^n(x), n\geq 1, x\in \C_p, x\ne \hat x$ is
related to the real dynamical systems generated by $\varphi_{\alpha,\beta}$, $\phi_{\alpha,\beta}$ and $\psi_\alpha$. Now we are going to study these (real) dynamical systems.

\begin{lemma}\label{l1a} The dynamical system generated by $\varphi_{\alpha,\beta}(r), \alpha<\beta$ has the following properties:
\begin{itemize}
\item[1.] ${\rm Fix}(\varphi_{\alpha,\beta})=\{0\}\cup$
$$\left\{\begin{array}{lll}
\{r: r>\beta\} \cup\{\beta^*:\, {\rm if}\, \beta=\beta^*\}, \, {\rm for} \, |c|_p=1,\\[2mm]
\{r: r<\alpha;\, {\rm if}\, \alpha=|c|_p \beta\}\cup\{\alpha:\, {\rm if}\, \alpha^*=|c|_p\alpha\}
\cup\{|c|_p\beta: \, {\rm if}\, \alpha<|c|_p\beta\} , \, {\rm for} \,|c|_p<1,\\[2mm]
\{\beta:\, {\rm if} \, \beta^*=|c|_p\beta\}, \, {\rm for} \,|c|_p>1,\\
\end{array}
\right.;$$

\item[2.] For $|c|_p=1$, independently on $\alpha$, we have
$$\lim_{n\to\infty}\varphi_{\alpha,\beta}^n(r)=\left\{\begin{array}{lll}
0, \ \ \mbox{for all} \ \ r<\beta,\\[2mm]
r, \ \ \mbox{for all} \ \ r>\beta,\\[2mm]
\beta^*, \ \ \mbox{if} \ \ r=\beta
\end{array}\right.;
$$

\item[3.] If $|c|_p>1$, then
\begin{itemize}
 \item[3.a)] If $\beta^*\ne |c|_p\beta$, then
$$\lim_{n\to\infty}\varphi_{\alpha,\beta}^n(r)=0,\ \ \mbox{for any}\ \ r\geq 0;$$

\item[3.b)] If $\beta^*=|c|_p\beta$, then
$$\lim_{n\to\infty}\varphi_{\alpha,\beta}^n(r)=\left\{\begin{array}{ll}
0,\ \ \mbox{for any}\ \ r\notin B=\{|c|^k_p\beta: k=0,1,2,\dots\},\\[2mm]
\beta, \ \ \mbox{if}\ \ r\in B\\
\end{array}\right.;
$$
\end{itemize}
\item[4.] If $|c|_p<1$, then
\begin{itemize}
\item[4.a)]
If $\alpha<|c|_p\beta$, then $$\lim_{n\to\infty}\varphi_{\alpha,\beta}^n(r)=\left\{\begin{array}{lll}
0, \ \ \mbox{for all} \ \ r<|c|_p\beta,\\[2mm]
r, \ \ \mbox{for} \ \ r=|c|_p\beta,\\[2mm]
+\infty, \ \ \mbox{if} \ \ r>|c|_p\beta
\end{array}\right.;
$$

\item[4.b)]
If $\alpha=|c|_p\beta$, $\alpha^*=|c|_p\alpha$, then $$\lim_{n\to\infty}\varphi_{\alpha,\beta}^n(r)=\left\{\begin{array}{lll}
r, \ \ \mbox{for all} \ \ r\leq\alpha,\\[2mm]
+\infty, \ \ \mbox{if} \ \ r>\alpha
\end{array}\right.;
$$

\item[4.c)]
If $\alpha=|c|_p\beta$, $\alpha^*\ne|c|_p\alpha$, then $$\lim_{n\to\infty}\varphi_{\alpha,\beta}^n(r)=\left\{\begin{array}{lll}
r, \ \ \mbox{for all} \ \ r<\alpha,\\[2mm]
\alpha^*/|c|_p, \ \ \mbox{for} \ \ r=\alpha,\\[2mm]
+\infty, \ \ \mbox{if} \ \ r>\alpha
\end{array}\right.;
$$

\item[4.d)]
If $\alpha>|c|_p\beta$, $\alpha^*=|c|_p\alpha$, then $$\lim_{n\to\infty}\varphi_{\alpha,\beta}^n(r)=\left\{\begin{array}{lll}
0, \ \ \mbox{for} \ \ r=0,\\[2mm]
\alpha, \ \ \mbox{for} \ \ r\in L=\{(|c|_p\beta)^k\alpha^{1-k}, \, k\geq 1\},\\[2mm]
+\infty, \ \ \mbox{if} \ \ r\notin L.
\end{array}\right.
$$

\item[4.e)]
If $\alpha>|c|_p\beta$, $\alpha^*>|c|_p\alpha$, then $$\lim_{n\to\infty}\varphi_{\alpha,\beta}^n(r)=\left\{\begin{array}{lll}
0, \ \ \mbox{for} \ \ r=0,\\[2mm]
+\infty, \ \ \mbox{if} \ \ r>0.
\end{array}\right.
$$

\item[4.f)]
If $\alpha>|c|_p\beta$, $\alpha^*<|c|_p\alpha$, then there exists $k\geq 2$ such that
the sequence
$$\mathcal C=\{\alpha, \varphi_{\alpha,\beta}(\alpha),\dots, \varphi^{k-1}_{\alpha,\beta}(\alpha)\}$$
is a $k$-cycle of $\varphi_{\alpha,\beta}$ and

$$\lim_{n\to\infty}\varphi_{\alpha,\beta}^n(r)=\left\{\begin{array}{lll}
0, \ \ \mbox{for} \ \ r=0,\\[2mm]
\in \mathcal C, \ \ \mbox{for} \ \ r\in U=\{r: \exists n\in\N, \varphi_{\alpha,\beta}^n(r)\in \mathcal C\},\\[2mm]
+\infty, \ \ \mbox{if} \ \ r\notin U.
\end{array}\right.
$$
\end{itemize}
\end{itemize}
\end{lemma}
\begin{proof} 1. This is the result of a simple analysis of the equation $\varphi_{\alpha,\beta}(r)=r$.

Proofs of parts 2-4 follow from the property that $\varphi_{\alpha,\beta}(r)$, $r\ne \alpha,\beta$ is an increasing
function.
\end{proof}

\begin{lemma}\label{lp} The dynamical system generated by $\phi_{\alpha,\beta}(r), \alpha>\beta$ has the following properties:
\begin{itemize}
\item[A.] ${\rm Fix}(\phi_{\alpha,\beta})=\{0\}\cup$
$$\left\{\begin{array}{lll}
\{r: r>\alpha\} \cup\{\alpha:\, {\rm if}\, \alpha=\alpha'\}, \, {\rm for} \, |c|_p=1,\\[2mm]
\{\alpha: \, {\rm if}\, \alpha'=|c|_p \alpha\},\, {\rm for} \,|c|_p<1,\\[2mm]
\{\alpha/|c|_p\}, \, {\rm if} \, \alpha>|c|_p\beta, \,|c|_p>1,\\[2mm]
\{\beta: \, {\rm if}\, \beta'=|c|_p\beta\}, {\rm for} \, \alpha<|c|_p\beta, \,|c|_p>1.
\end{array}
\right.;$$

\item[B.] For $|c|_p=1$, we have
\begin{itemize}
\item[B.a)] If $\alpha=\alpha'$, then
$$\lim_{n\to\infty}\phi_{\alpha,\beta}^n(r)=\left\{\begin{array}{lll}
0, \ \ \mbox{for} \ \ r=0,\\[2mm]
\alpha, \ \ \mbox{for all} \ \ r\leq\alpha,\\[2mm]
r, \ \ \mbox{for all} \ \ r>\alpha,\\[2mm]
\end{array}\right.;
$$

\item[B.b)] If $\alpha\ne\alpha'$, then
there exists $k\geq 2$ such that the sequence $$\alpha, \alpha'=\phi_{\alpha,\beta}(\alpha), \phi^2_{\alpha,\beta}(\alpha),\dots, \phi^{k-1}_{\alpha,\beta}(\alpha)$$ is a $k$-cycle for $\phi_{\alpha,\beta}(r)$ and  $\lim_{n\to\infty}\phi_{\alpha,\beta}^n(r)$ converges to the $k$-cycle for any $r\leq \alpha$. Moreover the limit is equal to $r$ for any $r>\alpha$.
\end{itemize}
\item[C.] If $|c|_p<1$, then
\begin{itemize}
 \item[C.a)] If $\alpha'=|c|_p\alpha$, then
$$\lim_{n\to\infty}\phi_{\alpha,\beta}^n(r)=\left\{\begin{array}{lll}
0,\ \ \mbox{if}\ \ r=0,\\[2mm]
\alpha,\ \ \mbox{for any}\ \ r\leq \alpha,\\[2mm]
+\infty,\ \ \mbox{for any}\ \ r>\alpha;
\end{array}
\right.$$
\item[C.b)] If $\alpha'\ne |c|_p\alpha$, then
there exists $k\geq 2$ such that the sequence
$$\mathbb P=\{\alpha, \alpha'=\phi_{\alpha,\beta}(\alpha), \phi^2_{\alpha,\beta}(\alpha),\dots, \phi^{k-1}_{\alpha,\beta}(\alpha)\}$$ is a $k$-cycle for $\phi_{\alpha,\beta}(r)$ and  $$\lim_{n\to\infty}\phi_{\alpha,\beta}^n(r)=\left\{\begin{array}{lll}
0, \, {\rm if} \ r=0,\\[2mm]
\in \mathbb P,\, {\rm if} \ r\in W=\{r: \exists n\in \N, \phi^{n}_{\alpha,\beta}(r)\in \mathbb P\},\\[2mm]
+\infty, \, {\rm if} \ r\notin W
\end{array}\right.;
$$
\end{itemize}
\item[D.] If $|c|_p>1$, then
\begin{itemize}
\item[D.a)]
If $\alpha>|c|_p\beta$, then $$\lim_{n\to\infty}\phi_{\alpha,\beta}^n(r)=\left\{\begin{array}{ll}
{\alpha\over |c|_p}, \ \ \mbox{for all} \ \ r>0,\\[2mm]
0, \ \ \mbox{for} \ \ r=0,
\end{array}\right.;
$$

\item[D.b)]
If $\alpha<|c|_p\beta$, $\beta'=|c|_p\beta$, then $$\lim_{n\to\infty}\phi_{\alpha,\beta}^n(r)=\left\{\begin{array}{ll}
\beta, \ \ \mbox{for all} \ \ r\in M=\{|c|_p^k\beta, k\geq 0\},\\[2mm]
0, \ \ \mbox{for} \ \ r\notin M,
\end{array}\right.;
$$

\item[D.c)]
If $\alpha<|c|_p\beta$, $\beta'\ne|c|_p\beta$, then $$\lim_{n\to\infty}\phi_{\alpha,\beta}^n(r)=
0, \ \ \mbox{for} \ \ r\geq 0.$$
\end{itemize}
\end{itemize}
\end{lemma}
\begin{proof} Since $\phi_{\alpha,\beta}(r)$ is a piecewise linear function the proof consists of simple computations, using the graph of the function and varying the parameters $\alpha,\beta, |c|_p$.
\end{proof}
The following lemma is obvious:

\begin{lemma}\label{lp1} The dynamical system generated by $\psi_{\alpha}(r), \alpha\geq 0$ has the following properties:
\begin{itemize}
\item[(i)] ${\rm Fix}(\psi_{\alpha})=\{0\}\cup
\left\{\begin{array}{ll}
\{r: r\ne\alpha\} \cup\{\alpha:\, {\rm if}\, \alpha=\hat\alpha\}, \, {\rm for} \, |c|_p=1,\\[2mm]
\{\alpha: \, {\rm if}\, \hat\alpha=|c|_p \alpha\},\, {\rm for} \,|c|_p\ne 1.
\end{array}
\right.;$

\item[(ii)] If $|c|_p=1$, then
$$\lim_{n\to\infty}\psi_{\alpha}^n(r)=\left\{\begin{array}{lll}
r,\ \ \mbox{for any}\ \ r\ne\alpha, \alpha\ne \hat\alpha,\\[2mm]
\hat\alpha, \ \ \mbox{for}\ \ r=\alpha, \alpha\ne \hat\alpha,\\[2mm]
r,\ \ \mbox{for any}\ \ r\geq 0, \alpha=\hat\alpha,
\end{array}
\right.$$

\item[(iii)] If $|c|_p>1$, then
$$\lim_{n\to\infty}\psi_{\alpha}^n(r)=\left\{\begin{array}{lll}
0,\ \ \mbox{for any}\ \ r\geq 0, |c|_p\alpha\ne \hat\alpha,\\[2mm]
\alpha, \ \ \mbox{for}\ \ r\in H=\{|c|_p^k\alpha: k\geq 0\}, \,|c|_p\alpha=\hat\alpha,\\[2mm]
0, \ \ \mbox{for}\ \ r\notin H,\, |c|_p\alpha=\hat\alpha,\\[2mm]
\end{array}
\right.$$

\item[(iv)] If $|c|_p<1$, then
$$\lim_{n\to\infty}\psi_{\alpha}^n(r)=\left\{\begin{array}{llll}
0, \ \ \mbox{for} \ \ r=0,\\[2mm]
+\infty,\ \ \mbox{for any}\ \ r>0, |c|_p\alpha\ne \hat\alpha,\\[2mm]
\alpha, \ \ \mbox{for}\ \ r\in H=\{|c|_p^k\alpha: k\geq 0\}, \,|c|_p\alpha=\hat\alpha,\\[2mm]
+\infty, \ \ \mbox{for}\ \ r\notin H,\, |c|_p\alpha=\hat\alpha.\\[2mm]
\end{array}
\right.$$
\end{itemize}
\end{lemma}

Now we shall apply these lemmas to the study of the $p$-adic dynamical system generated by $f$.

For $x\in S_{\alpha_i}(x_i)$, we denote
$$\alpha^*_i(x)=\alpha_i\cdot{|\alpha(x_i)+x-x_i|_p\over |\beta(x_i)+x-x_i|_p}, \ \ i=1,2.
$$
For $x\in S_{\beta_i}(x_i)$, we denote
$$\beta^*_i(x)=\beta_i\cdot{|\alpha(x_i)+x-x_i|_p\over |\beta(x_i)+x-x_i|_p}, \ \ i=1,2.
$$
Using Lemma \ref{lf} and Lemma \ref{l1a} we obtain the following

 \begin{thm}\label{t1a} If $\alpha_i<\beta_i$ and $x\in S_r(x_i)$, $i=1,2$, then
 the $p$-adic dynamical system generated by $f$ has the following properties:
\begin{itemize}
\item[1.] The following spheres are invariant with respect to $f$:
$$\begin{array}{lll}
S_r(x_i),\, {\rm if} \,r>\beta_i,\,  |c|_p=1,\\[2mm]
S_r(x_i), \, {\rm if} \, r<\alpha_i, \,\alpha_i=|c|_p \beta_i;|c|_p<1,\\[2mm]
S_{|c|_p\beta_i}(x_i), \, {\rm if}\, \alpha_i<|c|_p\beta_i,|c|_p<1.\\[2mm]
\end{array}
;$$

\item[2.] For $|c|_p=1$, we have
$$\lim_{n\to\infty}f^n(x)=x_i, \ \ \mbox{for all} \ \ r<\beta_i,$$
$$f\left(S_r(x_i)\setminus\mathcal P_3\right)\subset S_r(x_i), \ \ \mbox{for all} \ \ r>\beta_i,$$
$$\lim_{n\to \infty}f^n(x)\in S_{\beta_i^*(x)}(x_i), \ \ \mbox{if} \ \ r=\beta_i;$$

\item[3.] If $|c|_p>1$, then
\begin{itemize}
 \item[3.a)] If $\beta_i^*(x)\ne |c|_p\beta_i$, then
$$\lim_{n\to\infty}f^n(x)=x_i,\ \ \mbox{for any}\ \ r\geq 0;$$

\item[3.b)] If $\beta^*_i(x)=|c|_p\beta_i$, then
$$\lim_{n\to\infty}f^n(x)=\left\{\begin{array}{ll}
x_i,\ \ \mbox{for any}\ \ r\notin B,\\[2mm]
\in S_{{\beta^*_i(x)\over |c|_p}}(x_i), \ \ \mbox{if}\ \ r\in B\\
\end{array}\right.;
$$
\end{itemize}
\item[4.] If $|c|_p<1$, then
\begin{itemize}
\item[4.a)]
If $\alpha_i<|c|_p\beta_i$, then
$$\lim_{n\to\infty}f^n(x)=x_i, \ \ \mbox{for all} \ \ r<|c|_p\beta_i,$$
$$
f\left(S_r(x_i)\setminus \mathcal P_3\right)\subset S_r(x_i), \ \ \mbox{for} \ \ r=|c|_p\beta_i,$$
$$\lim_{n\to\infty}|f^n(x)-x_i|_p=+\infty, \ \ \mbox{if} \ \ r>|c|_p\beta_i;
$$

\item[4.b)]
If $\alpha_i=|c|_p\beta_i$, $\alpha^*_i(x)=|c|_p\alpha_i$, then $$
f\left(S_r(x_i)\setminus\mathcal P_3\right)\subset S_r(x_i),\  \ \mbox{for all} \ \ r\leq\alpha_i,$$
$$\lim_{n\to\infty}|f^n(x)-x_i|_p=+\infty,\ \ \mbox{if} \ \ r>\alpha_i;
$$

\item[4.c)]
If $\alpha_i=|c|_p\beta_i$, $\alpha^*_i(x)\ne|c|_p\alpha_i$, then

$$f\left(S_r(x_i)\setminus\mathcal P_3\right)\subset S_r(x_i), \ \ \mbox{for all} \ \ r<\alpha_i,$$
$$ \lim_{n\to\infty}f^n(x)\in S_{\alpha^*_i(x)/|c|_p}(x_i), \ \ \mbox{for} \ \ r=\alpha_i,$$
$$\lim_{n\to\infty}|f^n(x)-x_i|_p=+\infty, \ \ \mbox{if} \ \ r>\alpha_i;
$$

\item[4.d)]
If $\alpha_i>|c|_p\beta_i$, $\alpha_i^*=|c|_p\alpha_i$, then $$
\lim_{n\to\infty}f^n(x)\in S_{\alpha^*_i(x)/|c|_p}(x_i), \ \ \mbox{for} \ \ r\in L=\{(|c|_p\beta_i)^k\alpha_i^{1-k}, \, k\geq 1\},$$
$$
\lim_{n\to\infty}|f^n(x)-x_i|_p=+\infty, \ \ \mbox{if} \ \ r\notin L.
$$

\item[4.e)]
If $\alpha_i>|c|_p\beta_i$, $\alpha_i^*(x)>|c|_p\alpha_i$, then
$$
\lim_{n\to\infty}|f^n(x)-x_i|_p=+\infty, \ \ \mbox{if} \ \ r>0.
$$

\item[4.f)]
If $\alpha_i>|c|_p\beta_i$, $\alpha_i^*<|c|_p\alpha_i$, then there exists $k\geq 2$ such that
the limiting trajectory of $f^n(x), n\geq 1$ will periodically visit the spheres $S_{\varphi^j_{\alpha_i,\beta_i}(\alpha_i)}(x_i)$, $j=0,1,\dots,k-1$ in the following way: if $r\in U=\{r: \exists n\in\N, \varphi_{\alpha_i,\beta_i}^n(r)\in \mathcal C\}$, then
$$S_{\alpha_i}(x_i)\to S_{\varphi_{\alpha_i,\beta_i}(\alpha_i)}(x_i)\to\dots\to S_{\varphi^{k-1}_{\alpha_i,\beta_i}(\alpha_i)}(x_i)\to S_{\alpha_i}(x_i),$$ and
$$\lim_{n\to\infty}|f^n(x)-x_i|_p=+\infty, \ \ \mbox{if} \ \ r\notin U.$$
\end{itemize}
\end{itemize}
\end{thm}

By Lemma \ref{lf} and Lemma \ref{lp} we obtain the following

\begin{thm}\label{tlp} If $\alpha_i>\beta_i$ and $x\in S_r(x_i)$, $i=1,2$, then the $p$-adic dynamical system generated by $f$ has the following properties:
\begin{itemize}

\item[A.] The following spheres are invariant:
$$S_r(x_i), \ \ if \ \  r>\alpha_i,\ \ |c|_p=1,$$
$$S_{\alpha/|c|_p}(x_i), \, {\rm if} \, \alpha_i>|c|_p\beta_i, \,|c|_p>1.$$

\item[B.] For $|c|_p=1$, we have
\begin{itemize}

\item[B.a)] If $\alpha_i=\alpha^*_i(x)$, then
$$\lim_{n\to\infty}f^n(x)\in S_{\alpha_i^*(x)}(x_i), \ \ \mbox{for all} \ \ r\leq\alpha_i,$$
$$
f\left(S_r(x_i)\setminus\mathcal P_3\right)\subset S_r(x_i), \ \ \mbox{for all} \ \ r>\alpha_i;
$$

\item[B.b)] If $\alpha_i\ne\alpha_i^*(x)$, then
there exists $k\geq 2$ such that the limiting trajectory of $f^n(x), n\geq 1$ will periodically visit the spheres $S_{\phi^j_{\alpha_i,\beta_i}(\alpha_i)}(x_i)$, $j=0,1,\dots,k-1$ in the following way:
$$S_{\alpha_i}(x_i)\to S_{\phi_{\alpha_i,\beta_i}(\alpha_i)}(x_i)\to\dots\to S_{\phi^{k-1}_{\alpha_i,\beta_i}(\alpha_i)}(x_i)\to S_{\alpha_i}(x_i), \ \ for\, any \ \ r\leq \alpha_i$$ and
$$\lim_{n\to\infty}f^n(x)\in S_r(x_1), \ \ for\, any \ \ r>\alpha_i.$$
\end{itemize}

\item[C.] If $|c|_p<1$, then
\begin{itemize}

 \item[C.a)] If $\alpha_i^*(x)=|c|_p\alpha_i$, then
$$
f\left(S_r(x_i)\setminus\mathcal P_3\right)\subset S_\alpha(x_i),\ \ \mbox{for any}\ \ r\leq \alpha_i,$$
$$\lim_{n\to\infty}|f^n(x)-x_i|_p=+\infty,\ \ \mbox{for any}\ \ r>\alpha_i;$$

\item[C.b)] If $\alpha_i^*(x)\ne |c|_p\alpha_i$, then
there exists $k\geq 2$ such that the limiting trajectory of $f^n(x), n\geq 1$ will periodically visit the spheres $S_{\phi^j_{\alpha_i,\beta_i}(\alpha_i)}(x_i)$, $j=0,1,\dots,k-1$ in the following way: if $r\in W=\{r: \exists n\in \N, \phi^{n}_{\alpha_i,\beta_i}(r)\in \mathbb P\}$, then
$$S_{\alpha_i}(x_i)\to S_{\phi_{\alpha_i,\beta_i}(\alpha_i)}(x_i)\to\dots\to S_{\phi^{k-1}_{\alpha_i,\beta_i}(\alpha_i)}(x_i)\to S_{\alpha_i}(x_i),  $$
$$\lim_{n\to\infty}|f^n(x)-x_i|_p=+\infty,\ \ {\rm if} \ \ r\notin W;
$$
\end{itemize}
\item[D.] If $|c|_p>1$, then
\begin{itemize}
\item[D.a)]
If $\alpha_i>|c|_p\beta_i$, then $$\lim_{n\to\infty}f^n(x)\in S_{\alpha\over |c|_p}(x_1), \ \ \mbox{for all} \ \ r>0,$$

\item[D.b)]
If $\alpha_i<|c|_p\beta_i$, $\beta^*_i(x)=|c|_p\beta_i$, then $$\lim_{n\to\infty}f^n(x)\in
S_{\beta_i}(x_i), \ \ \mbox{for all} \ \ r\in M=\{|c|_p^k\beta_i, k\geq 0\},$$
$$\lim_{n\to\infty}f^n(x)=x_i, \ \ \mbox{for} \ \ r\notin M;$$

\item[D.c)]
If $\alpha_i<|c|_p\beta_i$, $\beta^*_i(x)\ne|c|_p\beta_i$, then $$\lim_{n\to\infty}f^n(x)=
x_i, \ \ \mbox{for} \ \ r\geq 0.$$
\end{itemize}
\end{itemize}
\end{thm}

By Lemma \ref{lf} and Lemma \ref{lp1} we get

\begin{thm}\label{tlp1} If $\alpha_i=\beta_i$, and $x\in S_r(x_i)$, $i=1,2$, then the dynamical system generated by $f$ has the following properties:
\begin{itemize}
\item[(i)] For any $r\ne \alpha_i$ the sphere $S_r(x_i)$ is an invariant set.

\item[(ii)] If $|c|_p=1$, then
\begin{itemize}
\item[(ii.a)] If $\alpha_i\ne\alpha_i^*(x)$, then
$$
f\left(S_r(x_i)\setminus\mathcal P_3\right)\subset S_r(x_i), \ \ \mbox{for all} \ \ r\ne\alpha_i;
$$
$$
f\left(S_r(x_i)\setminus\mathcal P_3\right)\subset S_{\alpha^*_i(x)}(x_i), \ \ \mbox{if} \ \ r=\alpha_i;
$$

\item[(ii.b)] If $\alpha_i=\alpha_i^*(x)$, then
$$
f\left(S_r(x_i)\setminus\mathcal P_3\right)\subset S_r(x_i), \ \ \mbox{for all} \ \ r\neq 0;
$$
\end{itemize}
\item[(iii)] If $|c|_p>1$, then
\begin{itemize}
\item[(iii.a)] If $\alpha^*_i(x)\ne |c|_p\alpha_i$, then
$$\lim_{n\to\infty}f^n(x)=x_i, \ \ \mbox{for any}\ \ r\geq 0;$$

\item[(iii.b)] If $\alpha^*_i(x)=|c|_p\alpha_i$, then
$$\lim_{n\to\infty}f^n(x)=x_i, \ \ \mbox{for any}\ \  r\in H=\{|c|_p^k\alpha_i: k\geq 0\},$$ $$ \lim_{n\to\infty}f^n(x)\in S_\alpha(x_i)\ \ \mbox{for}\ \ r\notin H.$$
\end{itemize}
\item[(iv)] If $|c|_p<1$, then
\begin{itemize}
\item[(iv.a)] If $\alpha^*_i(x)\ne |c|_p\alpha_i$, then
$$\lim_{n\to\infty}|f^n(x)-x_i|_p=+\infty,\ \ \mbox{for any}\ \ r>0,$$

\item[(iv.b)] If $\alpha^*_i(x)=|c|_p\alpha_i$, then if $r\in H=\{|c|_p^k\alpha_i: k\geq 0\}$, then
$$\lim_{n\to\infty}f^n(x)\in S_{\alpha_i}(x_i),\ \ \mbox{for any}\ \ r\in H,$$
$$\lim_{n\to\infty}|f^n(x)-x_i|_p=+\infty,\ \ \mbox{for any}\ \ r\notin H.$$
\end{itemize}
\end{itemize}
\end{thm}

\section*{Acknowledgments}

U.Rozikov thanks the universit\'e du Sud Toulon Var for supporting his visit to Toulon
and the Centre de Physique Th\'eorique de Marseille for kind hospitality. He also would like to acknowledge
the hospitality of the "Institut f\"{u}r Angewandte Mathematik",
Universit\"{a}t Bonn (Germany). This work is supported in part by
the DFG AL 214/36-1 project (Germany).

\end{document}